\documentclass[reqno,12pt]{article}

\usepackage{amsmath,amsfonts,amsthm,amssymb,indentfirst,mathrsfs}

\usepackage{empheq}

\usepackage{verbatim}
\usepackage{graphicx}

\usepackage{tikz}
\usetikzlibrary{shapes,snakes}
\usetikzlibrary{shapes.misc,shadows}

\usepackage{enumerate}
\usepackage[all]{xy}

\usepackage{color}

\newcommand{\EE}{E}

 \newtheorem{prop}{Proposition}[section]
 \newtheorem{lemma}[prop]{Lemma}

 \newtheorem{theorem}[prop]{Theorem}

\numberwithin{equation}{section}

\newcommand{\Var}{{\rm Var \ \! }}
\newcommand{\Cov}{{\rm \ \! Cov \ \! }}

\newcommand{\real}{\mathbb R}

\allowdisplaybreaks

\newcommand*{\Scale}[2][4]{\scalebox{#1}{$#2$}}

\begin{document}

\title{\huge Supermodular ordering of binomial, Poisson and Gaussian random vectors by tree-based correlations 
} 

\author{B\"unyamin K\i z\i ldemir\footnote{bunyamin001@e.ntu.edu.sg}
 \qquad \quad
Nicolas Privault\footnote{nprivault@ntu.edu.sg}
\\
\normalsize
        Division of Mathematical Sciences\\
\normalsize
        School of Physical and Mathematical Sciences\\
\normalsize
        Nanyang Technological University\\
\normalsize
        637371 Singapore
        }

\date{\today}  

\maketitle

\vspace{-1cm}
\begin{abstract} 
 We construct a tree-based dependence structure for
 the representation of binomial, Poisson and Gaussian
 random vectors having a given covariance matrix, 
 using sums of independent random variables.
 This construction allows us to characterize the
 supermodular ordering of such random vectors
 via the componentwise ordering of their covariance matrices.
 Our method relies on the representation of dependent components using binary trees on the discrete $d$-dimensional hypercube $C_d$, and on M\"obius inversion techniques. In the case of Poisson random vectors this approach involves L\'evy measures on $C_d$, and it is consistent with the approximation of Poisson and multivariate Gaussian random vectors by binomial vectors.
\end{abstract} 

\baselineskip0.671cm

\section{Introduction}
 A $d$-dimensional random vector
 $X = (X_1, \ldots , X_d)$ is said
 to be dominated by another random vector
 $Y = (Y_1, \ldots , Y_d)$ in the supermodular order,
 and one writes $X \le_{\rm sm} Y$, if
\begin{equation*}
 E [ \Phi ( X ) ] \leq E [ \Phi ( Y ) ],
\end{equation*} 
 for all sufficiently integrable
 {\em supermodular} functions, i.e. for all functions 
 $\Phi : \real^d \longrightarrow \real$ such that
$$
 \Phi ( x ) + \Phi ( y )
 \leq
 \Phi ( x \wedge y ) + \Phi ( x \vee y ),
 \quad
 x , y \in \real^d,
$$
 where the maximum $\vee$ and the minimum $\wedge$
 are defined with respect to the componentwise order
 of $x , y \in \real^d$. 
 The supermodular stochastic ordering is used in particular to
 capture a preference for greater inter-dependence in economic variables.
\\
 
In the case where $X$ and $Y$ are multivariate Gaussian vectors,
supermodular ordering has been 
characterized by the componentwise
ordering of their covariance matrices 
in \cite{scarsini} Theorem~4.2, 
cf. also Theorem~3.13.5 of \cite{mullerbk}.
In \cite{kp}, a dependence
structure under which supermodular ordering
can be characterized under the same covariance condition
has been provided for Poisson random vectors, 
based on a decomposition of their L\'evy measures on the
vertices of the $d$-dimensional unit hypercube $C_d$.
\\
 
In this paper we construct a general dependence structure for
binomial and Poisson random vectors,
under which the supermodular ordering
can be characterized by the ordering of covariance matrices. 
This approach extends the results of \cite{kp} to a 
larger family of dependence structures, based on an arrangement of random
variables according to a binary tree and on M\"obius inversion. 
In the Gaussian case it allows us to represent any 
multivariate Gaussian vector using sums of independent
Gaussian random variables as in \eqref{ui} below, 
and to recover the result of \cite{scarsini} as a consequence. 
Similarly, in the binomial, gamma and Poisson settings it yields
a construction of random vectors 
having an arbitrarily given covariance matrix, 
and it provides the associated characterization of binomial and Poisson
supermodular ordering, cf. Theorems~\ref{fjkdsfds} and 
\ref{sufficiency}.
We refer the reader to \cite{hu-xie-ruan} and references therein
for the use of a different type of tree-based dependence 
in the setting of Bernoulli random vectors. 
\\

 We proceed as follows. 
 In Section~\ref{ds} we construct a general dependence structure
 that uses binary trees on the vertices of the $d$-dimensional hypercube. 
 In Section~\ref{s1} we apply this dependence structure
 to the characterization of the binomial 
 supermodular ordering via the componentwise ordering of covariances,
 cf. Theorem~\ref{fjkdsfds}. 
 In Section~\ref{s2} we deal with the case of 
 Poisson random vectors via the use of L\'evy measures on
 the vertices of the $d$-dimensional unit hypercube $C_d$,
 cf. Theorem~\ref{sufficiency}. 
 This also includes extensions to the increasing supermodular order,
 cf. Proposition~\ref{djklaaqe}.
 This result naturally extends to the supermodular
 ordering of sums of binomial,
 multivariate Gaussian and Poisson random vectors. 
 We also include a remark on the related convex 
 ordering problem for Poisson random vectors
 in Proposition~\ref{djklasasaa}.
\section{Tree-based correlation structures} 
\label{ds}
In this section we introduce the general
dependence structure used in this paper. 
Let $(e_1,\ldots , e_d)$ denote the canonical basis of $\real^d$, 
and let
$$
C_d : = \{ 0 , 1 \}^d
= \big\{
x = (x_1,\ldots , x_d ) \ : \
x_i\in \{ 0,1 \}, \ i=1,\ldots , d \big\}
$$ 
denote the set of vertices of the $d$-dimensional unit hypercube.
\\

We identify $C_d$ to the
power set $\{0,1 \}^d \simeq \{ S \in \{ 1,\ldots , d \}\}$ of
$\{ 1,\ldots , d \}$, i.e. 
each $x_S =(x_1,\ldots , x_d) \in C_d$ is identified to its index set
$S = \big\{
i \in \{ 1, \ldots , d \} \ : \ x_i = 1\big\}$.
In particular, we write 
$a \in x=(x_1,\ldots , x_d)$ when $x_a=1$, and $x\setminus \{ a \}$ for
$(x_i {\bf 1}_{\{ i \not= a \}})_{i=1,\ldots , d}$.
\\

We also endow $C_d$ 
  with the natural inclusion ordering of index sets
  i.e. we write $x \preceq y$ when $x\subseteq y$,
  or equivalently $0\leq x_i \leq y_i \leq 1$,
  $i=1,\ldots , d$, 
  and $x \prec y$ when $x \preceq y$ and $x\not= y$,
  i.e. $x\subsetneq y$. 
\subsubsection*{Random vectors} 
Given $( e_{k,l} )_{1\leq  k \leq l \leq  d}$ a  family of elements of $C_d$ 
with $e_{k,k} : = e_k$, $k=1,\ldots , d$, 
and a family $(X_{i,j})_{1\leq  i \leq  j \leq  d}$ of 
 independent random variables,
consider the random vector $X = (X_1,\ldots , X_d)$ given by 
$$
X_i := 
\sum_{ i \in e_{k,l} } X_{k,l}, \qquad i =1,\ldots , d. 
$$ 
 In other words, we have 
$$ 
X  = \sum_{i=1}^d e_i \sum_{ i \in e_{k,l} } X_{k,l}
= \sum_{1 \leq  k \leq l \leq  d} X_{k,l} e_{k,l}. 
$$ 
 Denoting by $(m_{k,l})_{1\leq  k \leq  l \leq  d }$ and 
 $(\sigma^2_{k,l})_{1\leq  k \leq  l \leq  d }$
 the respective means and variances of
 $(X_{k,l})_{1\leq k \leq  l \leq  d}$, we have 
\begin{equation*} 
E [ X_i ] = \sum_{i \in e_{k,l} } m_{k,l}, 
\qquad i =1,\ldots , d, 
\end{equation*}
 and 
\begin{equation}
\label{fjkljljkl2} 
\Cov ( X_i, X_j )
= \sum_{i, j \in e_{k,l}} \Var [ X_{k,l} ] 
= \sum_{e_{i,j} \preceq e_{k,l} } \sigma^2_{k,l},
\qquad 1\leq i \leq j \leq d.
\end{equation}
\subsubsection*{Binary trees} 
From now on we work under the following assumption 
on $(e_{k,l} )_{1\leq k < l \leq d}$. 
\begin{enumerate}[(H)]
\item The family $(e_{k,l} )_{1\leq k \leq l \leq d}\subset C_d$ forms a binary 
  tree of size $d(d+1)/2$ 
  in which every node $e_{k,l}$ has two children 
 $e_{k,l} \Scale[0.7]{\setminus \{ k \}}$ and $e_{k,l} \Scale[0.7]{\setminus \{ l \}}$. 
\end{enumerate} 
In particular, $(e_{k,l} )_{1\leq k \leq l \leq d}$ forms a binary
tree with height at most $d$ for the partial order $\preceq$. 
\\
 
\noindent
{\bf Example.}
 When $d=5$, consider 
 
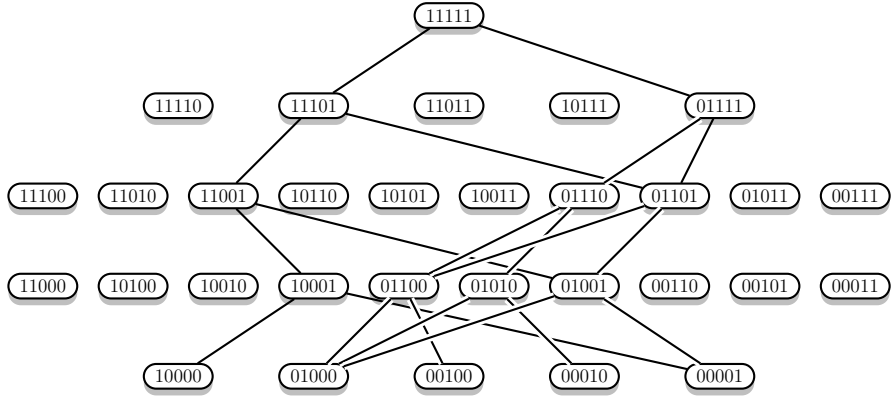
\begin{center}
   \begin{tikzpicture}[thick,scale=0.6, every node/.style={transform shape}]

\node(11111)[draw,fill=white,
   shape=rounded rectangle,
   drop shadow={opacity=.5,shadow xshift=0pt},
   minimum width=1.8cm]
   at (0,2) {11111};

\node(11110)[draw,fill=white,
   shape=rounded rectangle,
   drop shadow={opacity=.5,shadow xshift=0pt},
   minimum width=1.8cm]
   at (-6,0) {11110};
\node(11101)[draw,fill=white,
   shape=rounded rectangle,
   drop shadow={opacity=.5,shadow xshift=0pt},
   minimum width=1.8cm]
   at (-3,0) {11101};
\node(11011)[draw,fill=white,
   shape=rounded rectangle,
   drop shadow={opacity=.5,shadow xshift=0pt},
   minimum width=1.8cm]
   at (0,0) {11011};
\node(10111)[draw,fill=white,
   shape=rounded rectangle,
   drop shadow={opacity=.5,shadow xshift=0pt},
   minimum width=1.8cm]
   at (3,0) {10111};
\node(01111)[draw,fill=white,
   shape=rounded rectangle,
   drop shadow={opacity=.5,shadow xshift=0pt},
   minimum width=1.8cm]
   at (6,0) {01111};

\node(11100)[draw,fill=white,
   shape=rounded rectangle,
   drop shadow={opacity=.5,shadow xshift=0pt},
   minimum width=1.8cm]
   at (-9,-2) {11100};
\node(11010)[draw,fill=white,
   shape=rounded rectangle,
   drop shadow={opacity=.5,shadow xshift=0pt},
   minimum width=1.8cm]
   at (-7,-2) {11010};
\node(11001)[draw,fill=white,
   shape=rounded rectangle,
   drop shadow={opacity=.5,shadow xshift=0pt},
   minimum width=1.8cm]
   at (-5,-2) {11001};
\node(10110)[draw,fill=white,
   shape=rounded rectangle,
   drop shadow={opacity=.5,shadow xshift=0pt},
   minimum width=1.8cm]
   at (-3,-2) {10110};
\node(10101)[draw,fill=white,
   shape=rounded rectangle,
   drop shadow={opacity=.5,shadow xshift=0pt},
   minimum width=1.8cm]
   at (-1,-2) {10101};
\node(10011)[draw,fill=white,
   shape=rounded rectangle,
   drop shadow={opacity=.5,shadow xshift=0pt},
   minimum width=1.8cm]
   at (1,-2) {10011};
\node(01110)[draw,fill=white,
   shape=rounded rectangle,
   drop shadow={opacity=.5,shadow xshift=0pt},
   minimum width=1.8cm]
   at (3,-2) {01110};
\node(01101)[draw,fill=white,
   shape=rounded rectangle,
   drop shadow={opacity=.5,shadow xshift=0pt},
   minimum width=1.8cm]
   at (5,-2) {01101};
\node(01011)[draw,fill=white,
   shape=rounded rectangle,
   drop shadow={opacity=.5,shadow xshift=0pt},
   minimum width=1.8cm]
   at (7,-2) {01011};
\node(00111)[draw,fill=white,
   shape=rounded rectangle,
   drop shadow={opacity=.5,shadow xshift=0pt},
   minimum width=1.8cm]
   at (9,-2) {00111};

\node(11000)[draw,fill=white,
   shape=rounded rectangle,
   drop shadow={opacity=.5,shadow xshift=0pt},
   minimum width=1.8cm]
   at (-9,-4) {11000};
\node(10100)[draw,fill=white,
   shape=rounded rectangle,
   drop shadow={opacity=.5,shadow xshift=0pt},
   minimum width=1.8cm]
   at (-7,-4) {10100};
\node(10010)[draw,fill=white,
   shape=rounded rectangle,
   drop shadow={opacity=.5,shadow xshift=0pt},
   minimum width=1.8cm]
   at (-5,-4) {10010};
\node(10001)[draw,fill=white,
   shape=rounded rectangle,
   drop shadow={opacity=.5,shadow xshift=0pt},
   minimum width=1.8cm]
   at (-3,-4) {10001};
\node(01100)[draw,fill=white,
   shape=rounded rectangle,
   drop shadow={opacity=.5,shadow xshift=0pt},
   minimum width=1.8cm]
   at (-1,-4) {01100};
\node(01010)[draw,fill=white,
   shape=rounded rectangle,
   drop shadow={opacity=.5,shadow xshift=0pt},
   minimum width=1.8cm]
   at (1,-4) {01010};
\node(01001)[draw,fill=white,
   shape=rounded rectangle,
   drop shadow={opacity=.5,shadow xshift=0pt},
   minimum width=1.8cm]
   at (3,-4) {01001};
\node(00110)[draw,fill=white,
   shape=rounded rectangle,
   drop shadow={opacity=.5,shadow xshift=0pt},
   minimum width=1.8cm]
   at (5,-4) {00110};
\node(00101)[draw,fill=white,
   shape=rounded rectangle,
   drop shadow={opacity=.5,shadow xshift=0pt},
   minimum width=1.8cm]
   at (7,-4) {00101};
\node(00011)[draw,fill=white,
   shape=rounded rectangle,
   drop shadow={opacity=.5,shadow xshift=0pt},
   minimum width=1.8cm]
   at (9,-4) {00011};

\node(10000)[draw,fill=white,
   shape=rounded rectangle,
   drop shadow={opacity=.5,shadow xshift=0pt},
   minimum width=1.8cm]
   at (-6,-6) {10000};
\node(01000)[draw,fill=white,
   shape=rounded rectangle,
   drop shadow={opacity=.5,shadow xshift=0pt},
   minimum width=1.8cm]
   at (-3,-6) {01000};
\node(00100)[draw,fill=white,
   shape=rounded rectangle,
   drop shadow={opacity=.5,shadow xshift=0pt},
   minimum width=1.8cm]
   at (0,-6) {00100};
\node(00010)[draw,fill=white,
   shape=rounded rectangle,
   drop shadow={opacity=.5,shadow xshift=0pt},
   minimum width=1.8cm]
   at (3,-6) {00010};
\node(00001)[draw,fill=white,
   shape=rounded rectangle,
   drop shadow={opacity=.5,shadow xshift=0pt},
   minimum width=1.8cm]
   at (6,-6) {00001};

   \draw (11111) -- (11101);
   \draw (11111) -- (01111);

   \draw (11101) -- (01101);
   \draw (11101) -- (11001);

   \draw[preaction={draw=white, -,line width=3pt}] (01111) -- (01110);
   \draw (01111) -- (01101);

   \draw (11001) -- (10001);
   \draw (11001) -- (01001);

   \draw[preaction={draw=white, -,line width=3pt}] (01110) -- (01100);
   \draw[preaction={draw=white, -,line width=3pt}] (01110) -- (01010);

   \draw[preaction={draw=white, -,line width=3pt}] (01101) -- (01100);
   \draw[preaction={draw=white, -,line width=3pt}] (01101) -- (01001);

   \draw (10001) -- (10000);
   \draw (10001) -- (00001);

   \draw[preaction={draw=white, -,line width=3pt}] (01100) -- (01000);
   \draw[preaction={draw=white, -,line width=3pt}] (01100) -- (00100);

   \draw[preaction={draw=white, -,line width=3pt}] (01010) -- (01000);
   \draw[preaction={draw=white, -,line width=3pt}] (01010) -- (00010);

   \draw[preaction={draw=white, -,line width=3pt}] (01001) -- (01000);
   \draw[preaction={draw=white, -,line width=3pt}] (01001) -- (00001);

\end{tikzpicture}
\end{center}

with 

$$
\begin{array}{cc|ccccc|}
 e_{1,2} & = & 1 & 1 & 0 & 0 & 1
 \\
 e_{1,3} & = & 1 & 1 & 1 & 0 & 1
 \\
 e_{1,4} & = & 1 & 1 & 1 & 1 & 1
 \\
 e_{1,5} & = & 1 & 0 & 0 & 0 & 1
 \\
 e_{2,3} & = & 0 & 1 & 1 & 0 & 0
 \\
 e_{2,4} & = & 0 & 1 & 0 & 1 & 0
 \\
 e_{2,5} & = & 0 & 1 & 0 & 0 & 1
 \\
 e_{3,4} & = & 0 & 1 & 1 & 1 & 0
 \\
 e_{3,5} & = & 0 & 1 & 1 & 0 & 1
 \\
 e_{4,5} & = & 0 & 1 & 1 & 1 & 1
\end{array}
$$
 and 
$$
\left\{
\begin{array}{l}
\displaystyle X_1 = X_{1,1} + X_{1,2} + X_{1,3} + X_{1,4} + X_{1,5}
\\
\displaystyle X_2 = X_{2,2} + X_{1,2} + X_{1,3} + X_{1,4} + X_{2,3} + X_{2,4} + X_{2,5} + X_{3,4} + X_{3,5} + X_{4,5}
\\
\displaystyle X_3 = X_{3,3} + X_{1,3} + X_{1,4} + X_{2,3} + X_{3,4} + X_{3,5} + X_{4,5}
\\
\displaystyle X_4 = X_{4,4} + X_{1,4} + X_{2,4} + X_{3,4} + X_{4,5}
\\
\displaystyle X_5 = X_{5,5} + X_{1,2} + X_{1,3} + X_{1,4} + X_{1,5} + X_{2,5} + X_{3,5} + X_{4,5}. 
\end{array}
\right.
$$
\subsubsection*{M\"obius inversion} 
 By M\"obius inversion, cf. Proposition~2 of \cite{rota1964}
 or Proposition~2.6.3 of \cite{peccatitaqqu},
 the coefficients $(\sigma^2_{k,l})_{1\leq  k \leq  l \leq  d }$
 in \eqref{fjkljljkl2} can be recovered using the covariances
$(\Cov ( X_i, X_j ) )_{1\leq  i \leq  j \leq  d }$ as
\begin{equation}
\label{djhklsdf} 
 \sigma^2_{k,l} = 
 \sum_{e_{k,l} \preceq e_{i,j} }
 \mu ( e_{i,j} , e_{k,l} )
 \Cov ( X_i, X_j ) 
 , \qquad 1\leq k \leq l \leq d, 
\end{equation}
 where $\mu ( x , y )$ is the M\"obius function 
 defined recursively by $\mu ( x , x ) := 1$ and 
 \begin{equation}
  \label{mu} 
 \mu ( x , y ) = - \sum_{y \prec z \preceq x}
 \mu ( x , z ),
 \qquad
 x,y \in C_d,
\end{equation} 
 cf. Proposition~1 of \cite{rota1964}. 
 Given $e_{k,l}\in C_d$, 
 the two children $e_{k,l}\Scale[0.7]{\setminus \{ k \}}$,
 and $e_{k,l} \Scale[0.7]{\setminus \{ l \}}$
 of $e_{k,l}$ have themselves a unique common child
 $e_{k,l} \Scale[0.7]{\setminus \{ k,l \}}$, and \eqref{mu} yields 
\begin{subequations}
\begin{empheq}[left=\empheqlbrace]{align}
\label{a1} 
& \displaystyle
\mu ( e_{k,l} , e_{k,l} ) = 1,
  \\ 
  \nonumber
   \\ 
\nonumber 
   & \mu ( e_{k,l} , e_{k,l} \Scale[0.7]{\setminus \{ k \}} ) = -1, 
  \\ 
  \nonumber
  \\ 
\nonumber 
  & \mu ( e_{k,l} , e_{k,l} \Scale[0.7]{\setminus \{ l \}} ) = -1, 
  \\ 
  \nonumber
  \\ 
\label{a4} 
  & \mu ( e_{k,l} , e_{k,l} \Scale[0.7]{\setminus \{ k,l \}} ) = 1, 
\end{empheq}
\end{subequations}
\vskip0.2cm
  as shown in the next graph: 
 
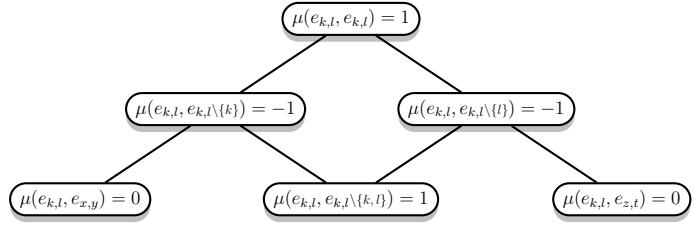
\begin{center}
  \begin{tikzpicture}[thick,scale=0.6, every node/.style={transform shape}]

\node(11111)[draw,fill=white,
  shape=rounded rectangle,
  drop shadow={opacity=.5,shadow xshift=0pt},
  minimum width=1.8cm]
  at (0,2) {$\mu ( e_{k,l} , e_{k,l} ) = 1$};

\node(11101)[draw,fill=white,
  shape=rounded rectangle,
  drop shadow={opacity=.5,shadow xshift=0pt},
  minimum width=1.8cm]
  at (-3,0) {$\mu ( e_{k,l} , e_{k,l} \Scale[0.7]{\setminus \{ k \}} ) = -1$};

\node(10111)[draw,fill=white,
  shape=rounded rectangle,
  drop shadow={opacity=.5,shadow xshift=0pt},
  minimum width=1.8cm]
  at (3,0) {$\mu ( e_{k,l} , e_{k,l} \Scale[0.7]{\setminus \{ l \}} ) = -1$};

\node(11001)[draw,fill=white,
  shape=rounded rectangle,
  drop shadow={opacity=.5,shadow xshift=0pt},
  minimum width=1.8cm]
  at (-6,-2) {$\mu ( e_{k,l} , e_{x,y} ) = 0$};

\node(10101)[draw,fill=white,
  shape=rounded rectangle,
  drop shadow={opacity=.5,shadow xshift=0pt},
  minimum width=1.8cm]
  at (0,-2) {$\mu ( e_{k,l} , e_{k,l} \Scale[0.7]{\setminus \{ k,l \}} ) = 1$};
\node(01101)[draw,fill=white,
  shape=rounded rectangle,
  drop shadow={opacity=.5,shadow xshift=0pt},
  minimum width=1.8cm]
  at (6,-2) {$\mu ( e_{k,l} , e_{z,t} ) = 0$};

  \draw (11111) -- (11101); 
  \draw (11111) -- (10111); 
  
  \draw (11101) -- (11001); 
  \draw (10111) -- (10101); 
  \draw (10111) -- (01101); 
  \draw (11101) -- (10101);

\end{tikzpicture}
\end{center}
\noindent
 In addition, we have $\mu ( e_{k,l} , e_{i,j} ) = 0$ in all other cases.
\\

In other words, \eqref{fjkljljkl2} can be solved recursively for
$(\sigma^2_{k,l})_{1\leq  k \leq  l \leq  d }$ given the data of 
$(\Cov ( X_i, X_j ))_{1\leq  i \leq  j \leq  d }$ by starting from
the equality $\Cov ( X_k, X_l ) = \sigma^2_{k,l}$
at a root $e_{k,l}$ of the tree, and then by moving down
the tree step by step until each leaf $e_i$. 
\\ 
 
\noindent
{\bf Examples.}
$i)$ Comonotonic vectors. The binary tree contains the node 
$e_{k,l}=$ \begin{tikzpicture}[thick,scale=0.6, every node/.style={transform shape}]
\node(11111)[draw,fill=white,
   shape=rounded rectangle,
   drop shadow={opacity=.5,shadow xshift=0pt},
   minimum width=1.8cm]
   at (0,2) {$111\cdots 11$};
   \end{tikzpicture}
as root, and all coefficients
$\sigma^2_{i,j}$ vanish for $(i,j)\not= (k,l)$,
which corresponds to the vector $(X_{k,l},X_{k,l},\ldots ,X_{k,l})$
with $\sigma^2_{k,l} = \Var [ X_{k,l}]$. 
\\ 
   
$ii)$ Pairwise dependence. The binary tree is reduced to the
$d$ leaves $e_1,\ldots , e_d$, and to their parents 
$$
 e_{k,l} =
 (0, \ldots, 0
 \underset{ \uparrow \atop k}{ ,1 , }
 0 , \ldots , 0 ,
 \underset{ \uparrow \atop l}{1}
 ,
 0 , \ldots , 0 ),
 \qquad
 1 \leq  k \leq  l \leq  d, 
$$
 as in the following example with $d=4$: 
 
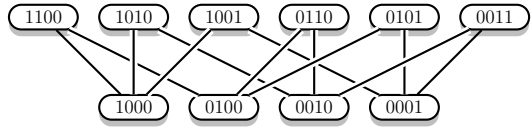
\begin{center}
   \begin{tikzpicture}[thick,scale=0.6, every node/.style={transform shape}]

\node(1100)[draw,fill=white,
   shape=rounded rectangle,
   drop shadow={opacity=.5,shadow xshift=0pt},
   minimum width=1.8cm]
   at (-5,-4) {1100};
\node(1010)[draw,fill=white,
   shape=rounded rectangle,
   drop shadow={opacity=.5,shadow xshift=0pt},
   minimum width=1.8cm]
   at (-3,-4) {1010};
\node(1001)[draw,fill=white,
   shape=rounded rectangle,
   drop shadow={opacity=.5,shadow xshift=0pt},
   minimum width=1.8cm]
   at (-1,-4) {1001};
\node(0110)[draw,fill=white,
   shape=rounded rectangle,
   drop shadow={opacity=.5,shadow xshift=0pt},
   minimum width=1.8cm]
   at (1,-4) {0110};
\node(0101)[draw,fill=white,
   shape=rounded rectangle,
   drop shadow={opacity=.5,shadow xshift=0pt},
   minimum width=1.8cm]
   at (3,-4) {0101};
\node(0011)[draw,fill=white,
   shape=rounded rectangle,
   drop shadow={opacity=.5,shadow xshift=0pt},
   minimum width=1.8cm]
   at (5,-4) {0011};

\node(1000)[draw,fill=white,
   shape=rounded rectangle,
   drop shadow={opacity=.5,shadow xshift=0pt},
   minimum width=1.8cm]
   at (-3,-6) {1000};
\node(0100)[draw,fill=white,
   shape=rounded rectangle,
   drop shadow={opacity=.5,shadow xshift=0pt},
   minimum width=1.8cm]
   at (-1,-6) {0100};
\node(0010)[draw,fill=white,
   shape=rounded rectangle,
   drop shadow={opacity=.5,shadow xshift=0pt},
   minimum width=1.8cm]
   at (1,-6) {0010};
\node(0001)[draw,fill=white,
   shape=rounded rectangle,
   drop shadow={opacity=.5,shadow xshift=0pt},
   minimum width=1.8cm]
   at (3,-6) {0001};

   \draw (1100) -- (1000);
   \draw[preaction={draw=white, -,line width=3pt}] (1100) -- (0100);

   \draw[preaction={draw=white, -,line width=3pt}] (1010) -- (1000);
   \draw[preaction={draw=white, -,line width=3pt}] (1010) -- (0010);

   \draw[preaction={draw=white, -,line width=3pt}] (1001) -- (1000);
   \draw[preaction={draw=white, -,line width=3pt}] (1001) -- (0001);

   \draw[preaction={draw=white, -,line width=3pt}] (0110) -- (0100);
   \draw[preaction={draw=white, -,line width=3pt}] (0110) -- (0010);

   \draw[preaction={draw=white, -,line width=3pt}] (0101) -- (0100);
   \draw[preaction={draw=white, -,line width=3pt}] (0101) -- (0001);

   \draw[preaction={draw=white, -,line width=3pt}] (0011) -- (0010);
   \draw (0011) -- (0001);

\end{tikzpicture}
\end{center}

 Here, the vector $(X_i)_{i=1,\ldots ,d}$ is given by 
\begin{equation}
\nonumber 
\left\{
\begin{array}{l}
\displaystyle X_1 = X_{1,1} + X_{1,2} + X_{1,3} + X_{1,4} 
\\
\displaystyle X_2 = X_{1.2} + X_{2,2} + X_{2,3} + X_{2,4} 
\\
\displaystyle X_3 = X_{1,3} + X_{2,3} + X_{3,3} + X_{3,4} 
\\
\displaystyle X_4 = X_{1,4} + X_{2,4} + X_{3,4} + X_{4,4}, 
\end{array}
\right.
\end{equation} 
 and for any $d\geq 1$, by \eqref{fjkljljkl2} we have 
 $$
 \Cov(X_i, X_j)=\sigma_{i,j}^2, \qquad
 1\le i< j \leq d,
 $$ 
 and
 $$
 \Var [X_i] = \sum_{j=1}^{i-1} \sigma^2_{j,i}+\sum_{j=i}^d \sigma^2_{i,j},
 \qquad i=1,\ldots , d.
 $$
\vskip0.4cm 
  
\noindent 
$iii)$ Gaussian vectors.
If $(U_1,\ldots , U_d)$ is
a centered multivariate Gaussian random vector 
with covariance matrix 
$(\Cov ( U_i, U_j ))_{1\leq  i \leq  j \leq  d }$ 
we can apply the M\"obius inversion 
 \eqref{djhklsdf} 
 in order to determine the variance coefficients
 $(\sigma^2_{k,l})_{1\leq  k \leq  l \leq  d }
 = (\Var [ Z_{k,l} ] )_{1\leq  k \leq  l \leq  d }$
 of independent centered Gaussian random variables such that 
 \begin{equation}
   \label{ui} 
 U_i := \sum_{ i \in e_{k,l} } Z_{k,l},
 \qquad
 i =1,\ldots ,d. 
 \end{equation}
\vskip0.4cm 

 \noindent 
$iv)$ 
 As in $(iii)$ above, binomial, Poisson and gamma random vectors
 having a given covariance matrix can be constructed
 by solving \eqref{djhklsdf} on a binary tree
 since those distributions are characterized by their variance
 parameters and they are stable by summation for a given
 scale parameter.
 However in this case the construction may not be unique
 depending on the chosen binary tree, as their joint
 distribution is not characterized by their  
 covariance matrices. 
\\

 \noindent 
 $v)$
 The dependence structure considered in \cite{kp}
for Poisson random vectors corresponds to the 
binary tree built on the $d(d-1)/2$ nodes 
$$
 e_{i,j} =
 (1, \ldots, 1 , 
 \underset{ \uparrow \atop i}{1}
 , 0 , \ldots , 0 , 
 \underset{ \uparrow \atop j}{1}
 ,
 0 , \ldots , 0 ),
 \qquad
 1 \leq  i < j \leq  d, 
$$
 and on the $d$ leaves $e_1,\ldots , e_d$.
\section{Binomial random vectors} 
\label{s1}
Consider $(Z_1,\ldots , Z_n)$ independent Bernoulli random
variables with parameter $p\in [0,1]$ and
$(A(e_{k,l}))_{1\leq k \leq l \leq d}$ a {\em partition} of
 $\{ 1, \ldots , n \}$.
Let $(X_{A(e_{k,l})})_{1\leq k \leq l \leq d}$ denote the family
of independent binomial random variables given by 
\begin{equation*} 
X_{A(e_{k,l})} := \sum_{i\in A(e_{k,l})} Z_i,
\qquad
1\leq k \leq l \leq d, 
\end{equation*}
 with 
$$m_{k,l} = E[X_{A(e_{k,l})} ] = p | A(e_{k,l}) |
\qquad 
1\leq k \leq l \leq d, 
$$
where $| A(e_{k,l}) |$ denotes the cardinality of $A(e_{k,l})$, and 
$$
\sigma^2_{k,l} = \Var [X_{A(e_{k,l})} ] = p q | A(e_{k,l}) |,
\qquad 
1\leq k \leq l \leq d, 
$$ 
where $q:=1-p$.
Let now  
$$
A_i := \bigcup_{i \in e_{k,l} } A(e_{k,l}), \qquad i =1,\ldots , d, 
$$ 
 and consider the vector $(X_{A_1},\ldots , X_{A_d})$ of
 binomial random variables defined by 
\begin{equation}
\label{x1} 
X_{A_i} : = \sum_{k \in A_i} Z_k
= \sum_{i \in e_{k,l} } X_{A(e_{k,l})},
\qquad i =1,\ldots , d. 
\end{equation}

\noindent
{\bf Example.}
 When $d=5$, the binary tree 
 
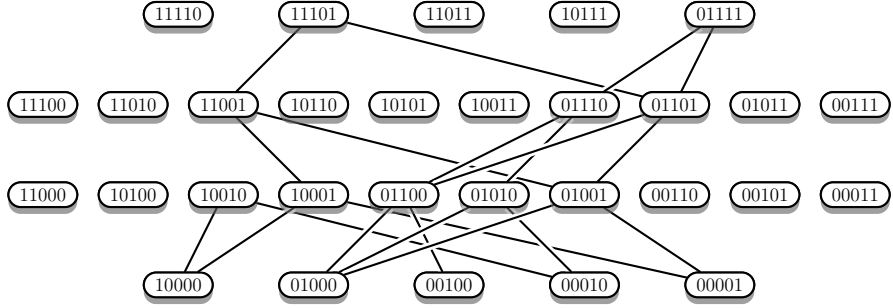
\begin{center}
   \begin{tikzpicture}[thick,scale=0.6, every node/.style={transform shape}]

\node(11110)[draw,fill=white,
   shape=rounded rectangle,
   drop shadow={opacity=.5,shadow xshift=0pt},
   minimum width=1.8cm]
   at (-6,0) {11110};
\node(11101)[draw,fill=white,
   shape=rounded rectangle,
   drop shadow={opacity=.5,shadow xshift=0pt},
   minimum width=1.8cm]
   at (-3,0) {11101};
\node(11011)[draw,fill=white,
   shape=rounded rectangle,
   drop shadow={opacity=.5,shadow xshift=0pt},
   minimum width=1.8cm]
   at (0,0) {11011};
\node(10111)[draw,fill=white,
   shape=rounded rectangle,
   drop shadow={opacity=.5,shadow xshift=0pt},
   minimum width=1.8cm]
   at (3,0) {10111};
\node(01111)[draw,fill=white,
   shape=rounded rectangle,
   drop shadow={opacity=.5,shadow xshift=0pt},
   minimum width=1.8cm]
   at (6,0) {01111};

\node(11100)[draw,fill=white,
   shape=rounded rectangle,
   drop shadow={opacity=.5,shadow xshift=0pt},
   minimum width=1.8cm]
   at (-9,-2) {11100};
\node(11010)[draw,fill=white,
   shape=rounded rectangle,
   drop shadow={opacity=.5,shadow xshift=0pt},
   minimum width=1.8cm]
   at (-7,-2) {11010};
\node(11001)[draw,fill=white,
   shape=rounded rectangle,
   drop shadow={opacity=.5,shadow xshift=0pt},
   minimum width=1.8cm]
   at (-5,-2) {11001};
\node(10110)[draw,fill=white,
   shape=rounded rectangle,
   drop shadow={opacity=.5,shadow xshift=0pt},
   minimum width=1.8cm]
   at (-3,-2) {10110};
\node(10101)[draw,fill=white,
   shape=rounded rectangle,
   drop shadow={opacity=.5,shadow xshift=0pt},
   minimum width=1.8cm]
   at (-1,-2) {10101};
\node(10011)[draw,fill=white,
   shape=rounded rectangle,
   drop shadow={opacity=.5,shadow xshift=0pt},
   minimum width=1.8cm]
   at (1,-2) {10011};
\node(01110)[draw,fill=white,
   shape=rounded rectangle,
   drop shadow={opacity=.5,shadow xshift=0pt},
   minimum width=1.8cm]
   at (3,-2) {01110};
\node(01101)[draw,fill=white,
   shape=rounded rectangle,
   drop shadow={opacity=.5,shadow xshift=0pt},
   minimum width=1.8cm]
   at (5,-2) {01101};
\node(01011)[draw,fill=white,
   shape=rounded rectangle,
   drop shadow={opacity=.5,shadow xshift=0pt},
   minimum width=1.8cm]
   at (7,-2) {01011};
\node(00111)[draw,fill=white,
   shape=rounded rectangle,
   drop shadow={opacity=.5,shadow xshift=0pt},
   minimum width=1.8cm]
   at (9,-2) {00111};

\node(11000)[draw,fill=white,
   shape=rounded rectangle,
   drop shadow={opacity=.5,shadow xshift=0pt},
   minimum width=1.8cm]
   at (-9,-4) {11000};
\node(10100)[draw,fill=white,
   shape=rounded rectangle,
   drop shadow={opacity=.5,shadow xshift=0pt},
   minimum width=1.8cm]
   at (-7,-4) {10100};
\node(10010)[draw,fill=white,
   shape=rounded rectangle,
   drop shadow={opacity=.5,shadow xshift=0pt},
   minimum width=1.8cm]
   at (-5,-4) {10010};
\node(10001)[draw,fill=white,
   shape=rounded rectangle,
   drop shadow={opacity=.5,shadow xshift=0pt},
   minimum width=1.8cm]
   at (-3,-4) {10001};
\node(01100)[draw,fill=white,
   shape=rounded rectangle,
   drop shadow={opacity=.5,shadow xshift=0pt},
   minimum width=1.8cm]
   at (-1,-4) {01100};
\node(01010)[draw,fill=white,
   shape=rounded rectangle,
   drop shadow={opacity=.5,shadow xshift=0pt},
   minimum width=1.8cm]
   at (1,-4) {01010};
\node(01001)[draw,fill=white,
   shape=rounded rectangle,
   drop shadow={opacity=.5,shadow xshift=0pt},
   minimum width=1.8cm]
   at (3,-4) {01001};
\node(00110)[draw,fill=white,
   shape=rounded rectangle,
   drop shadow={opacity=.5,shadow xshift=0pt},
   minimum width=1.8cm]
   at (5,-4) {00110};
\node(00101)[draw,fill=white,
   shape=rounded rectangle,
   drop shadow={opacity=.5,shadow xshift=0pt},
   minimum width=1.8cm]
   at (7,-4) {00101};
\node(00011)[draw,fill=white,
   shape=rounded rectangle,
   drop shadow={opacity=.5,shadow xshift=0pt},
   minimum width=1.8cm]
   at (9,-4) {00011};

\node(10000)[draw,fill=white,
   shape=rounded rectangle,
   drop shadow={opacity=.5,shadow xshift=0pt},
   minimum width=1.8cm]
   at (-6,-6) {10000};
\node(01000)[draw,fill=white,
   shape=rounded rectangle,
   drop shadow={opacity=.5,shadow xshift=0pt},
   minimum width=1.8cm]
   at (-3,-6) {01000};
\node(00100)[draw,fill=white,
   shape=rounded rectangle,
   drop shadow={opacity=.5,shadow xshift=0pt},
   minimum width=1.8cm]
   at (0,-6) {00100};
\node(00010)[draw,fill=white,
   shape=rounded rectangle,
   drop shadow={opacity=.5,shadow xshift=0pt},
   minimum width=1.8cm]
   at (3,-6) {00010};
\node(00001)[draw,fill=white,
   shape=rounded rectangle,
   drop shadow={opacity=.5,shadow xshift=0pt},
   minimum width=1.8cm]
   at (6,-6) {00001};

   \draw (11101) -- (11001);
   \draw (11101) -- (01101);

   \draw[preaction={draw=white, -,line width=3pt}] (01111) -- (01110);
   \draw (01111) -- (01101);

   \draw (11001) -- (10001);
   \draw (11001) -- (01001);

   \draw[preaction={draw=white, -,line width=3pt}] (01110) -- (01100);
   \draw[preaction={draw=white, -,line width=3pt}] (01110) -- (01010);

   \draw[preaction={draw=white, -,line width=3pt}] (01101) -- (01100);
   \draw (01101) -- (01001);

   \draw (10010) -- (10000);
   \draw (10010) -- (00010);

   \draw[preaction={draw=white, -,line width=3pt}] (10001) -- (10000);
   \draw (10001) -- (00001);

   \draw[preaction={draw=white, -,line width=3pt}] (01100) -- (01000);
   \draw[preaction={draw=white, -,line width=3pt}] (01100) -- (00100);

   \draw[preaction={draw=white, -,line width=3pt}] (01010) -- (01000);
   \draw[preaction={draw=white, -,line width=3pt}] (01010) -- (00010);

   \draw[preaction={draw=white, -,line width=3pt}] (01001) -- (01000);
   \draw (01001) -- (00001);

\node(11110)[draw,fill=white,
   shape=rounded rectangle,
   drop shadow={opacity=.5,shadow xshift=0pt},
   minimum width=1.8cm]
   at (-6,0) {11110};
\node(11101)[draw,fill=white,
   shape=rounded rectangle,
   drop shadow={opacity=.5,shadow xshift=0pt},
   minimum width=1.8cm]
   at (-3,0) {11101};
\node(11011)[draw,fill=white,
   shape=rounded rectangle,
   drop shadow={opacity=.5,shadow xshift=0pt},
   minimum width=1.8cm]
   at (0,0) {11011};
\node(10111)[draw,fill=white,
   shape=rounded rectangle,
   drop shadow={opacity=.5,shadow xshift=0pt},
   minimum width=1.8cm]
   at (3,0) {10111};
\node(01111)[draw,fill=white,
   shape=rounded rectangle,
   drop shadow={opacity=.5,shadow xshift=0pt},
   minimum width=1.8cm]
   at (6,0) {01111};

\node(11100)[draw,fill=white,
   shape=rounded rectangle,
   drop shadow={opacity=.5,shadow xshift=0pt},
   minimum width=1.8cm]
   at (-9,-2) {11100};
\node(11010)[draw,fill=white,
   shape=rounded rectangle,
   drop shadow={opacity=.5,shadow xshift=0pt},
   minimum width=1.8cm]
   at (-7,-2) {11010};
\node(11001)[draw,fill=white,
   shape=rounded rectangle,
   drop shadow={opacity=.5,shadow xshift=0pt},
   minimum width=1.8cm]
   at (-5,-2) {11001};
\node(10110)[draw,fill=white,
   shape=rounded rectangle,
   drop shadow={opacity=.5,shadow xshift=0pt},
   minimum width=1.8cm]
   at (-3,-2) {10110};
\node(10101)[draw,fill=white,
   shape=rounded rectangle,
   drop shadow={opacity=.5,shadow xshift=0pt},
   minimum width=1.8cm]
   at (-1,-2) {10101};
\node(10011)[draw,fill=white,
   shape=rounded rectangle,
   drop shadow={opacity=.5,shadow xshift=0pt},
   minimum width=1.8cm]
   at (1,-2) {10011};
\node(01110)[draw,fill=white,
   shape=rounded rectangle,
   drop shadow={opacity=.5,shadow xshift=0pt},
   minimum width=1.8cm]
   at (3,-2) {01110};
\node(01101)[draw,fill=white,
   shape=rounded rectangle,
   drop shadow={opacity=.5,shadow xshift=0pt},
   minimum width=1.8cm]
   at (5,-2) {01101};
\node(01011)[draw,fill=white,
   shape=rounded rectangle,
   drop shadow={opacity=.5,shadow xshift=0pt},
   minimum width=1.8cm]
   at (7,-2) {01011};
\node(00111)[draw,fill=white,
   shape=rounded rectangle,
   drop shadow={opacity=.5,shadow xshift=0pt},
   minimum width=1.8cm]
   at (9,-2) {00111};

\node(11000)[draw,fill=white,
   shape=rounded rectangle,
   drop shadow={opacity=.5,shadow xshift=0pt},
   minimum width=1.8cm]
   at (-9,-4) {11000};
\node(10100)[draw,fill=white,
   shape=rounded rectangle,
   drop shadow={opacity=.5,shadow xshift=0pt},
   minimum width=1.8cm]
   at (-7,-4) {10100};
\node(10010)[draw,fill=white,
   shape=rounded rectangle,
   drop shadow={opacity=.5,shadow xshift=0pt},
   minimum width=1.8cm]
   at (-5,-4) {10010};
\node(10001)[draw,fill=white,
   shape=rounded rectangle,
   drop shadow={opacity=.5,shadow xshift=0pt},
   minimum width=1.8cm]
   at (-3,-4) {10001};
\node(01100)[draw,fill=white,
   shape=rounded rectangle,
   drop shadow={opacity=.5,shadow xshift=0pt},
   minimum width=1.8cm]
   at (-1,-4) {01100};
\node(01010)[draw,fill=white,
   shape=rounded rectangle,
   drop shadow={opacity=.5,shadow xshift=0pt},
   minimum width=1.8cm]
   at (1,-4) {01010};
\node(01001)[draw,fill=white,
   shape=rounded rectangle,
   drop shadow={opacity=.5,shadow xshift=0pt},
   minimum width=1.8cm]
   at (3,-4) {01001};
\node(00110)[draw,fill=white,
   shape=rounded rectangle,
   drop shadow={opacity=.5,shadow xshift=0pt},
   minimum width=1.8cm]
   at (5,-4) {00110};
\node(00101)[draw,fill=white,
   shape=rounded rectangle,
   drop shadow={opacity=.5,shadow xshift=0pt},
   minimum width=1.8cm]
   at (7,-4) {00101};
\node(00011)[draw,fill=white,
   shape=rounded rectangle,
   drop shadow={opacity=.5,shadow xshift=0pt},
   minimum width=1.8cm]
   at (9,-4) {00011};

\node(10000)[draw,fill=white,
   shape=rounded rectangle,
   drop shadow={opacity=.5,shadow xshift=0pt},
   minimum width=1.8cm]
   at (-6,-6) {10000};
\node(01000)[draw,fill=white,
   shape=rounded rectangle,
   drop shadow={opacity=.5,shadow xshift=0pt},
   minimum width=1.8cm]
   at (-3,-6) {01000};
\node(00100)[draw,fill=white,
   shape=rounded rectangle,
   drop shadow={opacity=.5,shadow xshift=0pt},
   minimum width=1.8cm]
   at (0,-6) {00100};
\node(00010)[draw,fill=white,
   shape=rounded rectangle,
   drop shadow={opacity=.5,shadow xshift=0pt},
   minimum width=1.8cm]
   at (3,-6) {00010};
\node(00001)[draw,fill=white,
   shape=rounded rectangle,
   drop shadow={opacity=.5,shadow xshift=0pt},
   minimum width=1.8cm]
   at (6,-6) {00001};

\end{tikzpicture}
\end{center}
 corresponds to 
$$
\begin{array}{cc|ccccc|}
 e_{S(1,2)} & = & 1 & 1 & 0 & 0 & 1
 \\
 e_{S(1,3)} & = & 1 & 1 & 1 & 0 & 1
 \\
 e_{S(1,4)} & = & 1 & 0 & 0 & 1 & 0
 \\
 e_{S(1,5)} & = & 1 & 0 & 0 & 0 & 1
 \\
 e_{S(2,3)} & = & 0 & 1 & 1 & 0 & 0
 \\
 e_{S(2,4)} & = & 0 & 1 & 0 & 1 & 0
 \\
 e_{S(2,5)} & = & 0 & 1 & 0 & 0 & 1
 \\
 e_{S(3,4)} & = & 0 & 1 & 1 & 1 & 0
 \\
 e_{S(3,5)} & = & 0 & 1 & 1 & 0 & 1
 \\
 e_{S(4,5)} & = & 0 & 1 & 1 & 1 & 1
\end{array}
$$
and
$$
\hskip-0.2cm
\left\{
\hskip-0.2cm
\begin{array}{l}
\displaystyle X_{A_1} = X_{A(e_{1,1})} + X_{A(e_{1,2})} + X_{A(e_{1,3})} + X_{A(e_{1,4})} + X_{A(e_{1,5})} 
\\
\displaystyle X_{A_2} = X_{A(e_{2,2})} + X_{A(e_{1,2})} + X_{A(e_{1,3})} + X_{A(e_{2,3})} + X_{A(e_{2,4})} + X_{A(e_{2,5})}
+ X_{A(e_{3,4})}+ X_{A(e_{3,5})}+ X_{A(e_{4,5})}
\\
\displaystyle X_{A_3} = X_{A(e_{3,3})} + X_{A(e_{1,3})} + X_{A(e_{2,3})} + X_{A(e_{3,4})} + X_{A(e_{3,5})} + X_{A(e_{4,5})}
\\
\displaystyle X_{A_4} = X_{A(e_{4,4})} + X_{A(e_{1,4})} + X_{A(e_{2,4})} + X_{A(e_{3,4})} + X_{A(e_{4,5})}
\\
\displaystyle X_{A_5} = X_{A(e_{5,5})} + X_{A(e_{1,2})} + X_{A(e_{1,3})} + X_{A(e_{1,5})} + X_{A(e_{2,5})} + X_{A(e_{3,5})} + X_{A(e_{4,5})}. 
\end{array}
\right.
$$
 In general we have 
 $$
E [ X_{A_i} ] = p \sum_{i \in e_{k,l} } |A(e_{k,l})|, 
\qquad i =1,\ldots , d, 
$$ 
and
$$ 
 \Cov ( X_{A_i}, X_{A_j} ) = 
 p q \sum_{ e_{i,j} \preceq e_{k,l} } |A(e_{k,l})|, \qquad 1 \leq i \leq j \leq d. 
$$
 with the inversion formula 
$$ 
 p q |A(e_{k,l})| = 
 \sum_{e_{k,l} \preceq e_{i,j} }
 \mu ( e_{i,j} , e_{k,l} )
 \Cov ( X_i, X_j ) 
 , \qquad 1\leq k \leq l \leq d, 
$$ 
 that follows from \eqref{djhklsdf}. 
 Next is the main result of this section.
 \begin{theorem} 
 \label{fjkdsfds}
 Consider 
 $(X_{A_1},\ldots , X_{A_d})$ and 
 $(X_{B_1},\ldots , X_{B_d})$
 two binomial random vectors represented as
 in \eqref{x1}. 
The conditions
\begin{equation}
\label{con11}
\EE[X_{A_i}]=\EE[X_{B_i}],
 \qquad 1 \leq  i \leq  d,
\end{equation}
 and 
\begin{equation}
\label{con22}
\Cov(X_{A_i}, X_{A_j})\leq \Cov(X_{B_i}, X_{B_j}),
 \qquad 
 1 \leq i < j \leq d,
\end{equation}
 are necessary and sufficient for the
 supermodular ordering
 $$
 (X_{A_1},\ldots , X_{A_d})
 \le_{\rm sm}
 (X_{B_1},\ldots , X_{B_d}).
 $$ 
\end{theorem} 
\begin{proof}
 It is well-known, cf. e.g. Theorem~3.9.5 of \cite{mullerbk}, 
 that for any couple $(X,Y)$ of $d$-dimensional random vectors, the
 condition $X \le_{\rm sm} Y$ implies \eqref{con11} and
 \eqref{con22}, therefore it suffices to show sufficiency.
 Using induction, it is also sufficient to consider the case where
\begin{equation}
\label{hjkljk} 
 \Cov(X_{B_k}, X_{B_l}) = \Cov(X_{A_k} , X_{A_l}) + pq, 
\end{equation}
 for some given $1\leq k < l \leq d$, and 
\begin{equation}
\label{hjkljk2} 
 \Cov(X_{B_i}, X_{B_j} ) = \Cov (X_{A_i} , X_{A_j} ),
 \qquad 1 \leq i \leq j \leq d, \quad (i,j) \not= (k ,l ).
\end{equation} 
By the M\"obius inversion formula \eqref{djhklsdf}
there is a unique way (up to a permutation of $\{ 1,\ldots , n\}$)
to choose 
$(A(e_{k,l}))_{1\leq k \leq l \leq d}$ 
and
$(B(e_{k,l}))_{1\leq k \leq l \leq d}$ satisfying
\eqref{hjkljk} and \eqref{hjkljk2} respectively, with the relation
\begin{eqnarray*} 
  pq |B(e_{i,j})| & = & 
 \sum_{e_{i,j} \preceq e_{x,y}}
 \mu ( e_{x,y} , e_{i,j} )
 \Cov ( X_{B_x} , X_{B_y} ) 
 \\
  & = & 
 pq
 {\bf 1}_{\{ e_{i,j} \preceq e_{k ,l } \} }
 \mu ( e_{k ,l } , e_{i,j} )
 +
 \sum_{e_{i,j} \preceq e_{x,y}}
 \mu ( e_{x,y} , e_{i,j} )
 \Cov ( X_{A_x} , X_{A_y} ) \\
  & = & 
 pq {\bf 1}_{\{ e_{i,j} \preceq e_{k ,l } \} }
 \mu ( e_{k ,l } , e_{i,j} )
 + pq |A(e_{i,j})|,
 \qquad 1\leq i \leq j \leq d,
\end{eqnarray*} 
 from \eqref{djhklsdf}, i.e.
\begin{equation}
  \label{fjkg} 
|B(e_{i,j})|
=
{\bf 1}_{\{ e_{i,j} \preceq e_{k ,l } \} }
 \mu ( e_{k ,l } , e_{i,j} )
 + |A(e_{i,j})|,
 \qquad 1\leq i \leq j \leq d. 
\end{equation}
 Given the children
 $e_{k,l} \Scale[0.7]{\setminus \{ k \}}$, $e_{k,l} \Scale[0.7]{\setminus \{ l \}}\in C_d$
 and grandchild $e_{k,l} \Scale[0.7]{\setminus \{ k,l \}}$ 
 of $e_{k,l}\in C_d$, 
 by \eqref{a1}-\eqref{a4} and \eqref{fjkg} we have
$$
 \left\{
 \begin{array}{l}
 |B(e_{k,l})| = |A(e_{k,l})| + 1,
 \\ \\ 
 |B(e_{k,l} \Scale[0.7]{\setminus \{ k \}} )| = |A(e_{k,l} \Scale[0.7]{\setminus \{ k \}} )| - 1, 
 \\ \\
 |B(e_{k,l} \Scale[0.7]{\setminus \{ l \}} )| = |A(e_{k,l} \Scale[0.7]{\setminus \{ l \}} )| - 1, 
 \\ \\
 |B(e_{k,l} \Scale[0.7]{\setminus \{ k,l \}} )| = |A(e_{k,l} \Scale[0.7]{\setminus \{ k,l \}} )| + 1, 
 \end{array}
 \right.
$$ 
 with $| B(e_{i,j} ) | = | A(e_{i,j})|$
 in all other cases since $\mu ( e_{k,l} , e_{i,j} ) = 0$.
 We choose to realize the above as
 \begin{equation}
   \label{01} 
 \left\{
 \begin{array}{l}
 A(e_{k,l}) = B(e_{k,l}) \setminus \{ k \},
 \\ \\ 
 B(e_{k,l} \Scale[0.7]{\setminus \{ k \}} ) =
 A(e_{k,l} \Scale[0.7]{\setminus \{ k \}} ) \setminus \{ k \}, 
 \\ \\
 B(e_{k,l} \Scale[0.7]{\setminus \{ l \}} ) =
 A(e_{k,l} \Scale[0.7]{\setminus \{ l \}} ) \setminus \{ l \}, 
 \\ \\
 A(e_{k,l} \Scale[0.7]{\setminus \{ k,l \}} )
 =
 B(e_{k,l} \Scale[0.7]{\setminus \{ k,l \}} ) \setminus \{ l \}, 
 \end{array}
 \right.
 \end{equation}
 for some given $1\leq k < l \leq d$,
 with $k,l \notin B(e_{i,j} ) = A(e_{i,j})$ 
 in all other cases. 
 Noting that
$$ 
 l \in B(e_{k,l} \Scale[0.7]{\setminus \{ k,l \}} ), \quad
 k \in A(e_{k,l} \Scale[0.7]{\setminus \{ k \}} ), \quad
 l \in A(e_{k,l} \Scale[0.7]{\setminus \{ l \}} ),
$$
 and
 $$
 B(e_{k,l} \Scale[0.7]{\setminus \{ k,l \}} ) ) \cap B_k = \emptyset,
 \quad 
 B(e_{k,l} \Scale[0.7]{\setminus \{ k,l \}} ) ) \cap B_l = \emptyset,
 \quad 
 A(e_{k,l} \Scale[0.7]{\setminus \{ k \}} ) ) \cap A_k = \emptyset, 
 \quad
 A(e_{k,l} \Scale[0.7]{\setminus \{ l \}} ) ) \cap A_l = \emptyset,
 $$ 
 we find that 
 $$
 l \notin
 B_k, 
 \qquad 
 l \notin
 B_l, 
 \qquad 
 k\notin A_k, 
 \qquad
 l \notin A_l. 
$$ 
 Hence, using the symmetric difference operator $A\setminus B := A \cap B^c$,
 for $i=1,\ldots , d$ we have 
\begin{equation} 
\label{b1} 
 A_i =
 \left\{
 \begin{array}{ll}
( B_k \setminus B(e_{k,l})\setminus B(e_{k,l} \Scale[0.7]{\setminus \{ k,l \}} ) )
\cup A(e_{k,l}) 
  \cup \{ l \}, 
 & i = k, 
  \\
  \\
  ( B_i \setminus B(e_{k,l})\setminus B(e_{k,l} \Scale[0.7]{\setminus \{ k,l \}} ) 
) \cup A(e_{k,l}) \cup \{k\} 
\cup A(e_{k,l} \Scale[0.7]{\setminus \{ k,l \}} )
\cup \{ l \}, & i\notin \{ k, l\}, 
  \\
  \\
  ( B_l \setminus B(e_{k,l})\setminus B(e_{k,l} \Scale[0.7]{\setminus \{ k,l \}} ) )
\cup A(e_{k,l}) \cup \{ k \}, 
 & i = l,
 \end{array}
 \right. 
\end{equation}
\vskip0.2cm
 and 
\begin{equation}
\label{b2} 
 B_i =
 \left\{
 \begin{array}{ll}
( B_k \setminus B(e_{k,l})\setminus B(e_{k,l} \Scale[0.7]{\setminus \{ k,l \}} ) )
\cup B(e_{k,l}), 
 & i = k, 
  \\
  \\
  ( B_i \setminus B(e_{k,l})\setminus B(e_{k,l} \Scale[0.7]{\setminus \{ k,l \}} ) 
) \cup B(e_{k,l}) 
\cup B(e_{k,l} \Scale[0.7]{\setminus \{ k,l \}} ), & i\notin \{ k, l\}, 
  \\
  \\
 ( B_l \setminus B(e_{k,l})\setminus B(e_{k,l} \Scale[0.7]{\setminus \{ k,l \}} ) )
\cup B(e_{k,l}), 
 & i = l. 
\end{array}
 \right. 
\end{equation}
\vskip0.2cm
 In other words, from \eqref{01} we can write 
 \begin{equation}
   \label{001} 
 \left\{
 \begin{array}{l}
 X_{B(e_{k,l})} = X_{A(e_{k,l})} + U, 
 \\ \\ 
 X_{A(e_{k,l} \Scale[0.7]{\setminus \{ k \}} )} = X_{B(e_{k,l}\Scale[0.7]{\setminus \{ k \}} )} + U, 
 \\ \\
 X_{A(e_{k,l} \Scale[0.7]{\setminus \{ l \}} )} = X_{B(e_{k,l}\Scale[0.7]{\setminus \{ l \}} )} + V, 
 \\ \\
 X_{B(e_{k,l} \Scale[0.7]{\setminus \{ k,l \}} )} = X_{A(e_{k,l} \Scale[0.7]{\setminus \{ k,l \}} )} + V,
 \end{array}
 \right.
\end{equation}  
 where $U , V \in \{Z_1,\ldots , Z_n\}$
 are two independent Bernoulli random variables,
 while we have $X_{B(e_{i,j} )} = X_{A(e_{i,j})}$ in all other cases,
 and from \eqref{b1}-\eqref{b2} we get 
\begin{equation}
\label{c1} 
 X_{A_i} =
 \left\{
 \begin{array}{ll}
   X_{ B_k \setminus B(e_{k,l})\setminus B(e_{k,l} \Scale[0.7]{\setminus \{ k,l \}} ) }
    + X_{A(e_{k,l})} + V, 
 & i = k, 
  \\
  \\
  X_{B_i \setminus B(e_{k,l})\setminus B(e_{k,l} \Scale[0.7]{\setminus \{ k,l \} )}} 
 + X_{A(e_{k,l})} + U + X_{A(e_{k,l} \Scale[0.7]{\setminus \{ k,l \}} )}
 + V, & i\notin \{ k, l\}, ~~~~~~~
  \\
  \\
X_{B_l \setminus B(e_{k,l})\setminus B(e_{k,l} \Scale[0.7]{\setminus \{ k,l \}} )}
+ X_{A(e_{k,l})} + U, 
 & i = l,
 \end{array}
 \right. 
\end{equation}
\vskip0.2cm
and
\begin{equation}
\label{c2} 
 X_{B_i} =
 \left\{
 \begin{array}{ll}
   X_{ B_k \setminus B(e_{k,l})\setminus B(e_{k,l} \Scale[0.7]{\setminus \{ k,l \}} ) }
    + X_{B(e_{k,l})}, 
 & i = k, 
  \\
  \\
  X_{B_i \setminus B(e_{k,l})\setminus B(e_{k,l} \Scale[0.7]{\setminus \{ k,l \} ) }}
  + X_{B(e_{k,l})} + X_{B(e_{k,l} \Scale[0.7]{\setminus \{ k,l \}} )}, & i\notin \{ k, l\}. ~~~~~~~
  \\
  \\
    X_{ B_l \setminus B(e_{k,l})\setminus B(e_{k,l} \Scale[0.7]{\setminus \{ k,l \}} ) }
    + X_{B(e_{k,l})}, 
 & i = l. 
\end{array}
 \right. 
\end{equation}
\vskip0.2cm
Now, for any supermodular function $\phi : \real^d \longrightarrow \real$ 
 we have, using \eqref{c2} and \eqref{001}, 
 \begin{align*}
& 
     E \left[ \phi
     \left( \left( X_{B_i} \right)_{1\leq i \leq d} \right)
       \right]
   \\
   & = 
       E \left[ \phi
     \left(
     \left(
     X_{B_i \setminus B(e_{k,l})\setminus B(e_{k,l} \Scale[0.7]{\setminus \{ k,l \}} ) }
     +X_{B(e_{k,l})} + X_{B(e_{k,l} \Scale[0.7]{\setminus \{ k,l \}} )} {\bf 1}_{\{ i \notin \{ k , l \} \} }
     \right)_{1\leq i \leq d}   \right)
       \right]
   \\
   & = 
       E \left[ \phi
     \left(
     \left(
     X_{B_i \setminus B(e_{k,l})\setminus B(e_{k,l} \Scale[0.7]{\setminus \{ k,l \}} ) }
          +X_{A(e_{k,l})} + U 
     + ( X_{A(e_{k,l} \Scale[0.7]{\setminus \{ k,l \}} )}
     + V ) {\bf 1}_{\{ i \notin \{ k , l \} \} }
       \right)_{1\leq i \leq d} \right)
       \right]
   \\
   & \geq 
   E \left[ \phi
     \left(
     \left(
     X_{B_i \setminus B(e_{k,l})\setminus B(e_{k,l} \Scale[0.7]{\setminus \{ k,l \}} ) }
     + X_{A(e_{k,l})} + U {\bf 1}_{\{i \not= k \}}
     + X_{A(e_{k,l} \Scale[0.7]{\setminus \{ k,l \}} )}
     {\bf 1}_{\{ i \notin \{ k , l \} \} }
     + V {\bf 1}_{\{ i \not= l \}} 
       \right)_{1\leq i \leq d} \right)
       \right]
   \\
   & = 
   E \left[ \phi
     \left( \left( X_{A_i} \right)_{1\leq i \leq d} \right)
       \right], 
     \end{align*} 
 where we used \eqref{c1} for the last equality. 
 As for the above inequality, it follows from
 \begin{eqnarray*}
   \lefteqn{
     E \left[ \phi
     \left(
     U , U + V , \ldots , U+V , U ) 
     \right)        \right]
   }
   \\
   & = & 
   p^2  \phi
     \left(
     1 , 2 , \ldots , 2 , 1 
     \right)
       + q^2
       \phi
     \left(
     0, 0 , \ldots , 0 , 0  \right)
                 +
   pq \phi
     \left( 1 , 1 , \ldots , 1 , 1 \right)
       + pq 
     \phi
     \left( 0, 1 , \ldots , 1 , 0 \right)
   \\
   & \geq & 
   p^2  \phi
     \left(
     1 , 2 , \ldots , 2 , 1 
     \right)
       + q^2
       \phi
     \left(
     0, 0 , \ldots , 0 , 0  \right)
       + pq 
     \phi
     \left( 1, 1 , \ldots , 1 , 0 \right)
                 +
   pq \phi
     \left( 0 , 1 , \ldots , 1 , 1 \right)
        \\
     & = & 
     E \left[ \phi
     \left( U , U + V , \ldots , U+V , V \right)
       \right], 
\end{eqnarray*} 
 for all supermodular functions $\phi : \real^{|e_{k,l}|} \longrightarrow \real$,
 where $|e_{k,l}|$ denotes the cardinality of $e_{k,l}$
 whose indices are arranged as $\{ k,\ldots , l\}$
 for convenience of notation, 
 and we did not consider indices $j \notin e_{k,l}$
 as $U$ and $V$ do not belong to $X_j$ in this case. 
\end{proof}
\subsubsection*{Multivariate Gaussian vectors}
From the central limit theorem, 
Theorem~\ref{fjkdsfds} can be used to deal with centered
multivariate Gaussian random vectors
$(U_1,\ldots , U_d)$ and $(V_1,\ldots , V_d)$
with covariance matrices 
$$(\Cov ( U_i, U_j ))_{1\leq  i \leq  j \leq  d }
\quad \mbox{and} \quad 
(\Cov ( V_i, V_j ))_{1\leq  i \leq  j \leq  d }.
$$ 
 In this case we can apply the M\"obius inversion 
 \eqref{djhklsdf} 
 in order to determine the variance coefficients
 $(\sigma^2_{k,l})_{1\leq  k \leq  l \leq  d }$
 and 
 $(\eta^2_{k,l})_{1\leq  k \leq  l \leq  d }$
 in the decomposition \eqref{ui}.
 Those coefficients can then be  
 obtained as the respective limits of variances 
 $(\Var [ X^n_{k,l} /\sqrt{n} ] )_{1\leq  k \leq  l \leq  d }$
 and 
 $(\Var [ Y^n_{k,l} /\sqrt{n} ] )_{1\leq  k \leq  l \leq  d }$
 of independent binomial random variables such that
$$
 U^n_i := \frac{1}{\sqrt{n}}
 \sum_{ i \in e_{k,l} } ( X^n_{k,l} - E[ X^n_{k,l} ]) \quad \mbox{and} \quad 
 V^n_i := \frac{1}{\sqrt{n}}
 \sum_{ i \in e_{k,l} } ( Y^n_{k,l} - E[Y^n_{k,l}] ), \quad i =1,\ldots , d, 
$$ 
converge in distribution to 
$(U_1,\ldots , U_d)$ and $(V_1,\ldots , V_d)$
respectively.
Letting $n$ tend to infinity in $\Cov(U^n_i, U^n_j)\leq \Cov(V^n_i, V^n_j)$,
the condition $\Cov(U_i, U_j)\leq \Cov(V_i, V_j)$
of Theorem~\ref{fjkdsfds},
 $1 \leq i < j \leq d$, becomes necessary and sufficient
 for $(U_1,\ldots , U_d) \le_{\rm sm} (V_1,\ldots , V_d)$ 
 to hold. 
 In this way we recover the 
result of \cite{scarsini}, Theorem~4.2,
for Gaussian random vectors, 
cf. also Theorem~3.13.5 of \cite{mullerbk}.
\\

A similar argument results from Theorem~\ref{sufficiency} below
in the Poisson case, using
the convergence in distribution 
from renormalized binomial random variables
to Poisson random variables.  
In the next section we provide a proof
of such a result using L\'evy measures 
for infinitely divisible Poisson random vectors.
\section{Poisson random vectors}
\label{s2}
 Recall that any $d$-dimensional infinitely divisible
 Poisson random vector $X = (X_1,\ldots , X_d)$
 is defined by its characteristic function 
$$ 
 E [ e^{i\langle \bar{t} , X \rangle} ]
 = \exp \left( \int_{\real^d} ( e^{i\langle \bar{t} , x\rangle } - 1 ) \mu (dx) \right),$$
 where $\bar{t} = (t_1,\ldots , t_d) \in \real^d$, 
 $\langle \cdot , \cdot \rangle$ denotes the scalar product in $\real^d$,
 and the L\'evy measure 
 $$
 \mu (dx) :=
 \sum_{\emptyset \not= S \subset \{1,2,\ldots , d \} }
 a_S \delta_{e_S} (dx),
$$ 
 is supported on $C_d$, where $\delta_{e_S} (dx)$ denotes the
 Dirac measure at the point $e_S \in C_d$, and 
 $(a_S)_{\emptyset \not= S \subset \{1,2,\ldots , d \} }$
 is a family of nonnegative coefficients. 
 \\

 Equivalently, $X = (X_1,\ldots , X_d)$
 can be represented as
\begin{equation*}
 X_i
 = \sum_{ S \in \{ 0,1 \}^d \atop S \not= \emptyset }
 {\bf 1}_{\{ i \in S\} } X_S
 = \sum_{S \subset \{1,2,\ldots , d \} \atop S \ni i } X_S,
 \qquad
 i = 1,\ldots , d, 
\end{equation*}
 where $(X_S)_{\emptyset \not= S \subset \{1,2,\ldots , d \} }$
 is a family of $2^d -1$ independent Poisson random variables
 with respective intensities
 $(a_S)_{\emptyset \not= S \subset \{1,2,\ldots , d \} }$,
 cf. also Theorem~3 of \cite{kawamura}. 
\\
  
In order to characterize the ordering of Poisson random vectors
based on the data of their covariance matrices which contain
only $d(d+1)/2$ components,
we restrict ourselves to L\'evy measures of the form
\begin{equation}
\label{djkldd}
 \mu (dx) = \sum_{1 \leq  k \leq  l \leq  d} a_{k,l} \delta_{e_{k,l}} (dx),
\end{equation}
 on $C_d$, where $a_{k,l} \in \real_+$, $1\leq k \leq l \leq d$.
 In other words, we have 
 \begin{equation}
   \label{xp} 
X  = \sum_{i=1}^d e_i \sum_{ i \in e_{k,l} } X_{k,l}
= \sum_{1 \leq  k \leq l \leq  d} X_{k,l} e_{k,l}, 
 \end{equation}
 where $(X_{k,l})_{1\leq  k \leq  l \leq  d}$ is a family 
 of independent Poisson random variables whose respective
 intensity parameters $(a_{i,j})_{1\leq  i \leq  j \leq  d }$
 satisfy
 $\Var [ X_{k,l} ] = E[X_{k,l} ] = a_{k,l}$,
 $1\leq k \leq l \leq d$, with the inversion formula 
\begin{equation}
\label{djhklsdf.2} 
 a_{k,l} = 
 \sum_{e_{k,l} \preceq e_{i,j} }
 \mu ( e_{i,j} , e_{k,l} )
 \Cov ( X_i, X_j ) 
 , \qquad 1\leq k \leq l \leq d, 
\end{equation}
 that follows from \eqref{djhklsdf}. 
\subsubsection*{Supermodular ordering of Poisson random vectors}
Theorem~\ref{sufficiency} below is a direct consequence
of the following Lemma~\ref{djkld} which provides the
decomposition 
$$
 \mu ( dx )
 = 
 \sum_{i=1}^d \Var [X_i] \delta_{e_i}
 + \sum_{1 \leq i < j \leq  d } 
 \Cov(X_i, X_j)
 \left( \delta_{e_{ i,j }} + \delta_{e_{i,j} \Scale[0.7]{\setminus \{ i ,j\}}} 
 - \delta_{e_{i , j} \Scale[0.7]{\setminus \{ i \}} }
 - \delta_{e_{i,j} \Scale[0.7]{\setminus \{ j \}}  } 
 \right)
 $$ 
 of the L\'evy measure $\mu (dx)$ on $C_d \setminus \{0\}$,
 using the covariance matrix of $(X_i)_{i=1,\ldots ,d}$. 
\begin{lemma}
\label{djkld}
 For any function $\phi : C_d \longrightarrow \real$ such that
 $\phi ( 0 ) = 0$ we have 
\begin{eqnarray*} 
  \lefteqn{
   \! \! \! \! \!  \! \! \! \! \! 
    \int_{\real^d}
 \phi ( x ) \mu ( dx )
 = 
 \sum_{i=1}^d\EE[X_i]\phi(e_i )
  }
  \\
  & & 
\! \! \! \! \! \! \! \!  + 
 \sum_{1 \leq i < j \leq  d } 
 \Cov(X_i, X_j)
 ( \phi( e_{ i,j } ) + \phi( e_{i,j} \Scale[0.7]{\setminus \{ i ,j\}} ) 
 - \phi( e_{i , j} \Scale[0.7]{\setminus \{ i \}} )-\phi( e_{i,j}
 \Scale[0.7]{\setminus \{ j \}} ) )
. 
\end{eqnarray*} 
\end{lemma}
\begin{proof}
 By the M\"obius inversion formula \eqref{djhklsdf} 
 we have 
\begin{align*} 
&  \int_{\real^d}
 \phi ( x ) \mu ( dx )
  = 
 \sum_{1 \leq k \leq l \leq  d } a_{k,l}
 \phi(e_{k,l}) 
  \\
   & =  
 \sum_{1 \leq k \leq l \leq  d } 
 \phi(e_{k,l}) 
 \sum_{e_{k,l} \preceq e_{i,j} }
 \mu ( e_{i,j} , e_{k,l} )
 \Cov ( X_i, X_j ) 
 \\
   & =  
 \sum_{i=1}^d 
 \Cov ( X_i, X_i )
 \sum_{e_k \preceq e_i}
 \mu ( e_i , e_k )
 \phi(e_k) 
 +
 \sum_{1 \leq i < j \leq  d } 
 \Cov ( X_i, X_j )
 \sum_{e_{k,l} \preceq e_{i,j} \atop 1 \leq k < l \leq d}
 \mu ( e_{i,j} , e_{k,l} )
 \phi(e_{k,l}) 
 \\
  & =  
 \sum_{i=1}^d 
 E [X_i] \phi(e_i) 
 +
 \sum_{1 \leq i < j \leq  d } 
 \Cov(X_i, X_j)
 ( \phi( e_{i,j} ) + \phi(
     e_{i,j} \Scale[0.7]{\setminus \{ i,j \}} ) 
     - \phi( e_{i,j } \Scale[0.7]{\setminus \{ i \}} )
     - \phi( e_{i,j } \Scale[0.7]{\setminus \{ j \}} )), 
\end{align*} 
where we used \eqref{a1}-\eqref{a4} and the fact that
$e_k \preceq e_i$ if and only if $k=i$. 
\end{proof}
 Consider now two Poisson random vectors $X$ and $Y$ whose
 respective L\'evy measures $\mu$ and $\nu$ are represented as
$$
 \mu (dx) = \sum_{1 \leq  i \leq  j \leq  d} a_{i,j} \delta_{e_{i,j}} (dx)
 \quad
 \mbox{and}
 \quad
 \nu (dx) = \sum_{1 \leq  i \leq  j \leq  d} b_{i,j} \delta_{e_{i,j}} (dx),
$$
 as in \eqref{djkldd}.
 If $X_i$ has the same distribution as $Y_i$ for all $i=1,\ldots , d$
 then $E[X_i] = E[Y_i]$, $i=1,\ldots , d$, and
 Lemma~\ref{djkld} shows that 
\begin{eqnarray}
\label{djkldd1}
\lefteqn{
 \int_{\real^d}
 \phi ( y ) \nu ( dy )
 - 
 \int_{\real^d}
 \phi ( x ) \mu ( dx )
}
\\
\nonumber
 & = &
 \sum_{1\leq i < j \leq d}
 ( \Cov(Y_i, Y_j) - \Cov(X_i, X_j) )
  ( \phi( e_{i,j} ) + \phi(  e_{i,j} \Scale[0.7]{\setminus \{ i,j \}} ) 
 -\phi( e_{i,j} \Scale[0.7]{\setminus \{ i \}} )
 - \phi( e_{i,j} \Scale[0.7]{\setminus \{ j \}} ) )
. 
\end{eqnarray}
Relation~\eqref{djkldd1} shows in particular
that the nonnegativity of the coefficients
\begin{equation}
\label{cvy}
 \Cov(Y_i, Y_j) - \Cov(X_i, X_j) \geq 0, \qquad 1\leq i < j \leq d,
\end{equation}
 becomes a necessary and sufficient
 condition for the supermodular ordering of the
 (finite support) L\'evy measures $\mu$ and $\nu$. 
\\

The next Theorem~\ref{sufficiency} reformulates \eqref{cvy}
 as a necessary and sufficient condition for supermodular ordering
 of infinitely divisible Poisson random vector, 
 based on Theorem~4.5 of \cite{bauerle} which allows us to
 carry over the notion of supermodularity from the finite support
 setting of L\'evy measures $\mu$, $\nu$ on the cube $C_d$, to the
 infinite support setting of Poisson random variables.
\begin{theorem}
\label{sufficiency}
 Consider two Poisson random vectors $X$ and $Y$ both represented as
 in \eqref{xp}.
 Then the conditions
\begin{equation*}
\EE[X_i]=\EE[Y_i],
 \qquad
 1 \leq  i \leq  d,
\end{equation*}
 and 
\begin{equation}
\label{con2}
\Cov(X_i, X_j)\leq \Cov(Y_i, Y_j),
 \qquad
 1 \leq i < j \leq d,
\end{equation}
 are necessary and sufficient for the
 supermodular ordering $X \le_{\rm sm} Y$.
\end{theorem}
\begin{proof}
 By Theorem~4.5 in \cite{bauerle} it suffices to show that
\begin{equation} 
 \label{fdjklaa}
 \int_{\real^d}
 \phi (x) \mu (dx)
 \le
 \int_{\real^d}
 \phi (y) \nu (dy)
 \end{equation}
   for all supermodular functions
 $\phi : \real^d \longrightarrow \real$,
 where $\mu (dx)$ and $\nu (dy)$ respectively
 denote the L\'evy measures of $X$ and $Y$.
 By Lemma~\ref{djkld} we have the identity
\begin{eqnarray}
\nonumber
\lefteqn{
 \int_{\real^d}
 \phi ( y ) \nu ( dy )
 -
 \int_{\real^d}
 \phi ( x ) \mu ( dx )
}
\\
\nonumber
 & = &
 \sum_{1\leq i< j\leq d} ( \Cov(Y_i, Y_j) - \Cov(X_i, X_j) ) 
 ( \phi( e_{i,j} ) + \phi(  e_{i,j} \Scale[0.7]{\setminus \{ i,j \}} ) 
 -\phi( e_{ i ,j } \Scale[0.7]{\setminus \{ i \}} )
   -
   \phi( e_{ i,j} \Scale[0.7]{\setminus \{ j \}} ) )
\end{eqnarray} 
 under condition~\eqref{con2}, which allows us to conclude to 
 \eqref{fdjklaa} for all supermodular functions $\phi$.
\end{proof}
The next proposition is obtained as in Proposition~4.3 of \cite{kp}
by extending Theorem~4.5 of \cite{bauerle}
to nondecreasing supermodular functions
$\phi$ on $\real^d$ satisfying $\phi (0) = 0$,
using the same approximation as in Lemma~4.4 therein.
\begin{prop}
\label{djklaaqe}
 Consider two Poisson random vectors $X$ and $Y$ both represented as
 in \eqref{xp}, and assume that
$$
 \EE[X_i] \leq \EE[Y_i],
 \qquad
 1 \leq  i \leq  d,
$$
 and
$$
\Cov(X_i, X_j)\leq \Cov(Y_i, Y_j),
 \qquad
 1 \leq i < j \leq d.
$$
 Then we have
$$
 \EE[\Phi(X)] \leq \EE[\Phi(Y)]$$
 for all {\em nondecreasing} supermodular functions
 $\Phi : \real^d \longrightarrow \real$.
\end{prop}
\subsubsection*{Sums of binomial, Gaussian and Poisson vectors}
 By Theorem~4.2 of \cite{scarsini} on Gaussian random vectors,
 Theorems~\ref{fjkdsfds} and \ref{sufficiency} above,
 and the fact that the supermodular ordering is closed
 under convolution, cf. Theorem~3.9.14-(C) of \cite{mullerbk},
 deduce that the supermodular ordering
 of a sum of independent binomial, Gaussian and Poisson vectors,
 is implied by the componentwise ordering of their respective
 covariances. 
 Proposition~\ref{djklaaqe} admits an analog extension to
 sums of binomial, Gaussian and Poisson random vectors.
\subsubsection*{Convex ordering}
\begin{prop}
\label{djklasasaa}
 Consider two Poisson random vectors $X$ and $Y$ both represented as in
 \eqref{xp}.
 Then we have $X \le_{\rm cx} Y$ if and only if $X$ and $Y$ have
 same distributions.
\end{prop}
\begin{proof}
 Assume $X \le_{\rm cx} Y$, i.e. we have
$$
 \EE[\Phi(X)] \leq \EE[\Phi(Y)]$$
 for all {\em convex} functions
 $\Phi : \real^d \longrightarrow \real$. 
 Clearly, this implies $E[X_k] = E[Y_k]$, $k=1,\ldots ,d$.
 Next, choosing any $1 \leq k < l \leq d$
 we check that the function
$$
 ( x_1,\ldots , x_d ) \mapsto
 \phi_{k,l} ( x_1,\ldots , x_d ) := \max \left(
 0 ,
 x_l - x_k - \sum_{ a \notin e_{k,l}} x_a
 \right)
$$
 is convex on $\real^d$, with
 $\phi_{k,l} ( e_{i,j} ) = 1$
 when $e_{i,j}$ is a (non-strict) descendant of
 $e_{k,l} \Scale[0.7]{\setminus \{ k \}}$
 that contains $l$, 
 and $\phi_{k,l} ( e_{i,j} ) = 0$ in all other cases.
 This yields 
$$
 \phi_{k,l} (e_{k,l}) + \phi_{k,l} (e_{k,l} \Scale[0.7]{\setminus \{ k , l \}} )
 - \phi_{k,l} (e_{k,l}\Scale[0.7]{\setminus \{ k \}} ) - \phi_{k,l} (e_{k,l}\Scale[0.7]{\setminus \{ l \}} ) = -1,
$$
 and 
$$
 \phi_{k,l} (e_{i,j}) + \phi_{k,l} (e_{i,j} \Scale[0.7]{\setminus \{ i,j \}} ) - \phi_{k,l} (e_{i,j}\Scale[0.7]{\setminus \{ i \}} ) - \phi_{k,l} (e_{i,j}\Scale[0.7]{\setminus \{ j \}} ) = 0
$$
 when $(i,j) \not=(k,l)$.
 Hence by Lemma~\ref{djkld}, the condition
 $\Cov(Y_k, Y_l ) > \Cov(X_k, X_l )$ would imply 
\begin{eqnarray*}
\lefteqn{
 \int_{\real^d}
 \phi ( y ) \nu ( dy )
 -
 \int_{\real^d}
 \phi ( x ) \mu ( dx )
}
\\
\nonumber
 & = &
\sum_{1\leq i< j\leq d}
( \Cov(Y_i, Y_j) - \Cov(X_i, X_j) ) 
 ( \phi( e_{i,j} ) + \phi(  e_{i,j} \Scale[0.7]{\setminus \{ i,j \}} ) 
 -\phi( e_{ i ,j } \Scale[0.7]{\setminus \{ i \}} )
   -
   \phi( e_{ i,j} \Scale[0.7]{\setminus \{ j \}} ) )
\\
\nonumber
 & = &
( \Cov(Y_k, Y_l) - \Cov(X_k, X_l) ) 
 ( \phi( e_{k,l} ) + \phi(  e_{k,l} \Scale[0.7]{\setminus \{ k,l \}} ) 
 -\phi( e_{ k ,l } \Scale[0.7]{\setminus \{ k \}} )
   -
   \phi( e_{ k,l} \Scale[0.7]{\setminus \{ l \}} ) )
\\
 & < & 0,
\end{eqnarray*} 
which would contradict $X \le_{\rm cx} Y$ by the same argument
as in part $(b)$ of the proof of Theorem~4.5 in \cite{bauerle}.
 Hence we have $\Cov(Y_k, Y_l ) \leq \Cov(X_k, X_l )$, and
 proceeding similarly by exchanging $k$ and $l$ with the convex function
$$
 ( x_1,\ldots , x_d ) \mapsto
 \phi_{k,l} ( x_1,\ldots , x_d ) := \max \left(
 0 ,
 x_k - x_l + \sum_{ a \notin e_{k,l}} x_a
 \right)
$$
 we deduce that $\Cov(Y_k, Y_l) = \Cov(X_k, X_l)$
 for all $1 \leq k < l \leq d$, hence $X$ and $Y$ have same distribution
 from \eqref{djhklsdf.2}. 
\end{proof}

\footnotesize

\def\cprime{$'$} \def\polhk#1{\setbox0=\hbox{#1}{\ooalign{\hidewidth
  \lower1.5ex\hbox{`}\hidewidth\crcr\unhbox0}}}
  \def\polhk#1{\setbox0=\hbox{#1}{\ooalign{\hidewidth
  \lower1.5ex\hbox{`}\hidewidth\crcr\unhbox0}}} \def\cprime{$'$}

\end{document}